\documentclass[12pt]{amsart}
\usepackage{amsmath,latexsym,amsfonts,amssymb,amsthm}
\usepackage{geometry}
\usepackage{hyperref}
\usepackage{mathrsfs}
\usepackage{graphicx,color}
\usepackage[alphabetic]{amsrefs}
\usepackage[all,cmtip]{xy}
\usepackage{bbm}

\newtheorem{prop}{Proposition}[section]
\newtheorem{thm}[prop]{Theorem}
\newtheorem{cor}[prop]{Corollary}
\newtheorem{conj}[prop]{Conjecture}
\newtheorem{lem}[prop]{Lemma}

\theoremstyle{definition}
\newtheorem{que}[prop]{Question}
\newtheorem{defn}[prop]{Definition}
\newtheorem{expl}[prop]{Example}
\newtheorem{rem}[prop]{\it Remark}

\newtheorem*{claim*}{Claim}

\newcommand{\bP}{\mathbb{P}}

\newcommand{\bR}{\mathbb{R}}
\newcommand{\bA}{\mathbb{A}}
\newcommand{\bQ}{\mathbb{Q}}

\newcommand{\bN}{\mathbb{N}}

\newcommand{\bG}{\mathbb{G}}
\newcommand{\bT}{\mathbb{T}}
\newcommand{\bk}{\mathbbm{k}}

\newcommand{\oX}{\overline{X}}
\newcommand{\oY}{\overline{Y}}
\newcommand{\oDelta}{\overline{\Delta}}

\newcommand{\cX}{\mathcal{X}}

\newcommand{\cC}{\mathcal{C}}
\newcommand{\cO}{\mathcal{O}}
\newcommand{\cL}{\mathcal{L}}

\newcommand{\cJ}{\mathcal{J}}
\newcommand{\cE}{\mathcal{E}}
\newcommand{\cD}{\mathcal{D}}

\newcommand{\cR}{\mathcal{R}}
\newcommand{\cS}{\mathcal{S}}
\newcommand{\cW}{\mathcal{W}}

\newcommand{\ocD}{\overline{\mathcal{D}}}

\newcommand{\fa}{\mathfrak{a}}
\newcommand{\fb}{\mathfrak{b}}

\newcommand{\fm}{\mathfrak{m}}
\newcommand{\ft}{\mathfrak{t}}

\newcommand{\va}{\mathbf{a}}

\newcommand{\Spec}{\mathbf{Spec}}
\newcommand{\Proj}{\mathbf{Proj}}
\newcommand{\Supp}{\mathrm{Supp}}
\newcommand{\Hom}{\mathrm{Hom}}

\newcommand{\lct}{\mathrm{lct}}

\newcommand{\Pic}{\mathrm{Pic}}
\newcommand{\Aut}{\mathrm{Aut}}
\newcommand{\vol}{\mathrm{vol}}
\newcommand{\ord}{\mathrm{ord}}
\newcommand{\Val}{\mathrm{Val}}

\newcommand{\hvol}{\widehat{\rm vol}}

\newcommand{\wt}{\mathrm{wt}}
\newcommand{\Coef}{\mathrm{Coef}}
\newcommand{\Diff}{\mathrm{Diff}}

\newcommand{\mldk}{\mathrm{mld}^{\mathrm{K}}}

\newcommand{\gr}{\mathrm{gr}}

\numberwithin{equation}{section}

\title[Boundedness of singularities II]{On boundedness of singularities and minimal log discrepancies of Koll\'ar components, II}

\author{Ziquan Zhuang}

\address{Department of Mathematics, Johns Hopkins University, Baltimore, MD 21218, USA}
\email{zzhuang@jhu.edu}

\date{}

\begin{document}

\maketitle

\begin{abstract}
    We show that a set of K-semistable log Fano cone singularities is bounded if and only if their local volumes are bounded away from zero, and their minimal log discrepancies of Koll\'ar components are bounded from above. As corollaries, we confirm the boundedness conjecture for K-semistable log Fano cone singularities in dimension three, and show that local volumes of $3$-dimensional klt singularities only accumulate at zero.
\end{abstract}

\section{Introduction}

Following the recent study on K-stability of Fano varieties (see \cite{Xu-K-stability-survey} for a comprehensive survey), there has been growing interest in establishing a parallel stability theory for klt singularities, which are the local analog of Fano varieties. In the global theory, boundedness of Fano varieties plays an important role. Building on the seminal work of Birkar \cites{Birkar-bab-1,Birkar-bab-2}, Jiang \cite{Jia-Kss-Fano-bdd} proved that (in any fixed dimension) K-semistable Fano varieties with anti-canonical volumes bounded away from zero form a bounded family.\footnote{Several different proofs were later found in \cites{LLX-nv-survey,XZ-minimizer-unique}.} This is the first step in the general construction of the K-moduli space of Fano varieties; it is also a key ingredient in the proof \cite{LXZ-HRFG} of the global version of the Higher Rank Finite Generation Conjecture.

To advance the local stability theory, it is therefore natural to investigate the boundedness of klt singularities. Several years ago, Li \cite{Li-normalized-volume} introduced an interesting invariant of klt singularities called the local volume. It has become clear that the stability theory of klt singularities should be built around this invariant. In particular, it is speculated that (in any fixed dimension) klt singularities with local volumes bounded away from zero are specially bounded, i.e. they isotrivially degenerate to a bounded family. This is known in some special cases \cites{HLQ-vol-ACC,MS-bdd-toric,MS-bdd-complexity-one,Z-mld^K-1}, but the general situation is still quite mysterious. 

By the recent solution of the Stable Degeneration Conjecture \cites{Blu-minimizer-exist,LX-stability-higher-rank,Xu-quasi-monomial,XZ-minimizer-unique,LWX-metric-tangent-cone,XZ-SDC} (see also \cites{LX-stability-kc,BLQ-convexity}), every klt singularity has a canonical ``stable degeneration'' to a K-semistable log Fano cone singularity (see Section \ref{ss:log Fano cone} for the precise definition; roughly speaking, K-semistable log Fano cone singularities are generalizations of cones over K-semistable Fano varieties). This suggests a more precise boundedness conjecture (\cite[Conjecture 1.7]{XZ-SDC}, see also \cite[Problem 6.9]{SZ-no-semistability}): 

\begin{conj} \label{conj:K-ss cone bdd}
Fix $n\in \bN$, $\varepsilon>0$ and a finite set $I\subseteq [0,1]$. Then the set
\[
\Big\{
(X,\Delta)\  \Big|\begin{array}{l} x\in (X,\Delta) \mbox{ is a K-semistable log Fano cone singularity},\\ 
\dim X=n,\, {\rm Coef}(\Delta)\subseteq I,\, \mbox{and } \hvol(x,X,\Delta)\ge \varepsilon\end{array}
\Big\}
\]
is bounded.
\end{conj}

Here $\hvol(x,X,\Delta)$ denotes the local volume of the singularity $x\in (X,\Delta)$. This conjecture has been verified for toric singularities \cites{MS-bdd-toric,Z-mld^K-1}, for hypersurface singularities, and for singularities with torus actions of complexity one \cite{MS-bdd-complexity-one}.

In this paper, we study Conjecture \ref{conj:K-ss cone bdd} using the minimal log discrepancies of Koll\'ar components (or simply $\mldk$), see Section \ref{ss:kc}. Our main result gives a boundedness criterion in terms of $\mldk$ and the local volume.

\begin{thm}[=Corollary \ref{cor:K-ss cone bdd}] \label{main:K-ss cone bdd}
Fix $n\in \bN$ and consider a set $\cS$ of $n$-dimensional K-semistable log Fano cone singularities with coefficients in a fixed finite set $I\subseteq [0,1]$. Then $\cS$ is bounded if and only if there exist some $\varepsilon,A>0$ such that
\[
\hvol(x,X,\Delta)\ge \varepsilon \quad \mathit{and} \quad \mldk(x,X,\Delta)\le A
\]
for all $x\in (X,\Delta)$ in $\cS$.
\end{thm}

This upgrades the special boundedness result from our previous work \cite{Z-mld^K-1} to actual boundedness. We also prove boundedness results for more general log Fano cone singularities, replacing the K-semistability requirement by lower bounds on stability thresholds (as introduced in \cite{Hua-thesis}), see Corollary \ref{cor:cone bdd delta bdd below}.

Comparing Conjecture \ref{conj:K-ss cone bdd} and Theorem \ref{main:K-ss cone bdd}, we are naturally led to the following conjecture, already raised in \cite{Z-mld^K-1}:

\begin{conj}[{\cite[Conjecture 1.7]{Z-mld^K-1}}] \label{conj:bdd for mldK, vol>epsilon}
Let $n\in\bN$, $\varepsilon>0$ and let $I\subseteq [0,1]$ be a finite set. Then there exists some constant $A$ depending only on $n,\varepsilon,I$ such that 
\[
\mldk(x,X,\Delta)\le A
\]
for any $n$-dimensional klt singularity $x\in (X,\Delta)$ with $\Coef(\Delta)\subseteq I$ and $\hvol(x,X,\Delta)\ge \varepsilon$.
\end{conj}

This is known in dimension up to three \cite{Z-mld^K-1}. By Theorem \ref{main:K-ss cone bdd}, we then get a complete solution of Conjecture \ref{conj:K-ss cone bdd} in dimension three (the surface case was already treated in \cite{HLQ-vol-ACC}). The same result is also independently proved in \cite{LMS-bdd-dim-3} using a different method.

\begin{cor}[=Corollaries \ref{cor:3-dim Kss cone bdd}+\ref{cor:vol discrete in dim 3}]
Conjecture \ref{conj:K-ss cone bdd} holds in dimension $3$. Moreover, for any fixed finite coefficient set $I\subseteq [0,1]$, the set of possible local volumes of $3$-dimensional klt singularities is discrete away from zero.
\end{cor}

\subsection{Strategy of the proof}

In addition to our previous work \cite{Z-mld^K-1}, the proof of Theorem \ref{main:K-ss cone bdd} relies on several new ingredients. For simplicity, we assume $\Delta=0$, so that we are dealing with Fano cone singularities. Every such singularity $x\in X$ is an orbifold cone over some Fano variety $V$, so a natural idea is to prove Theorem \ref{main:K-ss cone bdd} by showing the boundedness of the associated Fano variety $V$. 

There are two main reasons why this na\"ive approach does not directly work. First, the orbifold base is highly non-unique; in fact, for a fixed Fano cone singularity the possible orbifold bases can be unbounded. For example, the simplest Fano cone singularity $0\in \bA^n$ can be realized as an orbifold cone over any weighted projective space of dimension $n-1$, but without further constraint, weighted projective spaces do not form a bounded family.

Moreover, even when there is a canonical choice of the orbifold base (e.g. when the singularity has a unique $\bG_m$-action), the anti-canonical volume of $V$ is only a fraction of the local volume of the singularity, where the factor is related to the Weil index of $V$ (defined as the largest number $q$ such that $-K_V\sim_\bQ qL$ for some Weil divisor $L$). In particular, the volume of $V$ can a priori be arbitrarily small, which needs to be ruled out if we want any boundedness of this sort.

Our solution is to turn the local boundedness question into a global one by considering the projective orbifold cone $\oX$ over $V$. A key observation is that by choosing the appropriate orbifold base $V$, the anti-canonical volume of $\oX$ approximates the local volume of $x\in X$ and in particular is bounded from both above and below. This takes care of the second issue mentioned above.

Still, as we make different choices of $V$, the corresponding projective orbifold cones $\oX$ can be unbounded. To circumvent this issue, we prove an effective birationality result for $\oX$. More precisely, we show that regardless of the choice of $V$, there is some positive integer $m$ depending only on $\hvol(x,X)$ and $\mldk(x,X)$ such that $|-mK_{\oX}|$ induces a birational map that restricts to an embedding on $X$. This is the main technical part of the proof, and ultimately reduces to the construction of certain isolated non-klt centers and some careful analysis of certain Izumi type constants, see Section \ref{ss:bir bdd}. Once the effective birationality is established, it is fairly straightforward to conclude the boundedness of $X$.

\subsection{Structure of the article}

In Sections \ref{ss:kc}--\ref{ss:bdd defn}, we give the necessary background on $\mldk$, local volumes, log Fano cone singularities and their boundedness. In sections \ref{ss:real coef}--\ref{ss:part I}, we collect some useful results in previous works; these include results from \cite{HLS-epsilon-plt-blowup} that tackles pairs with real coefficients, as well as modifications of some results from the prequel \cite{Z-mld^K-1} of the present work. In Section \ref{s:proof}, we present and prove a more general boundedness statement for polarized log Fano cone singularities, Theorem \ref{thm:polarized cone bdd}. Our main Theorem \ref{main:K-ss cone bdd} will follow as an application of Theorem \ref{thm:polarized cone bdd}.

\subsection*{Acknowledgement}

The author is partially supported by the NSF Grants DMS-2240926, DMS-2234736, a Clay research fellowship, as well as a Sloan fellowship. He would like to thank Harold Blum, J\'anos Koll\'ar, Yuchen Liu and Chenyang Xu for helpful discussions and comments. He also likes to thank the referees for careful reading of the manuscript and several helpful suggestions.

\section{Preliminaries}

\subsection{Notation and conventions}

We work over an algebraically closed field $\bk$ of characteristic $0$. We follow the standard terminology from \cites{KM98,Kol13}.

A pair $(X,\Delta)$ consists of a normal variety $X$ together with an effective $\bR$-divisor $\Delta$ on $X$ (a priori, we do not require that $K_X+\Delta$ is $\bR$-Cartier). A singularity $x\in (X,\Delta)$ consists of a pair $(X,\Delta)$ and a closed point $x\in X$. We will always assume that $X$ is affine and $x\in \Supp(\Delta)$ (whenever $\Delta\neq 0$). 

Suppose that $X$ is a normal variety. A prime divisor $F$ on some birational model $\pi\colon Y\to X$ (where $Y$ is normal and $\pi$ is proper) of $X$ is called a \emph{divisor over} $X$. Given an $\bR$-divisor $\Delta$ on $X$, we denote its strict transform on the birational model $Y$ by $\Delta_Y$. If $K_X+\Delta$ is $\bR$-Cartier, the log discrepancy $A_{X,\Delta}(F)$ is defined to be
\[
A_{X,\Delta}(F):=1+\ord_F(K_Y-\pi^*(K_X+\Delta)).
\]

A \emph{valuation} over a singularity $x\in X$ is an $\bR$-valued valuation $v: K(X)^* \to \bR$ (where $K(X)$ denotes the function field of $X$) such that $v$ is centered at $x$ (i.e. $v(f)>0$ if and only if $f\in \fm_x$) and $v|_{\bk^*}=0$. The set of such valuations is denoted as $\Val_{X,x}$. Given $\lambda\in\bR$, the corresponding \emph{valuation ideal} $\fa_\lambda (v)$ is
\[
\fa_\lambda (v):=\{f\in \cO_{X,x}\mid v(f)\ge \lambda\}. 
\]

When we refer to a constant $C$ as $C=C(n,\varepsilon,\cdots)$ it means $C$ only depends on $n,\varepsilon,\cdots$.

\subsection{Koll\'ar components} \label{ss:kc}

We first recall some definitions related to klt singularities and Koll\'ar components.

\begin{defn}[{\cite[Definition 2.34]{KM98}}]
We say a pair $(X,\Delta)$ is klt if $K_X+\Delta$ is $\bR$-Cartier, and for any prime divisor $F$ over $X$ we have $A_{X,\Delta}(F)>0$. We say $x\in (X,\Delta)$ is a klt singularity if $(X,\Delta)$ is klt.
\end{defn}

\begin{defn}[\cite{Xu-pi_1-finite}]
Let $x\in (X,\Delta)$ be a klt singularity and let $E$ be a prime divisor over $X$. If there exists a proper birational morphism $\pi\colon Y\to X$ such that $E= \pi^{-1}(x)$ is the unique exceptional divisor, $(Y,E+\Delta_Y)$ is plt and  $-(K_Y+\Delta_Y+E)$ is $\pi$-ample, we call $E$ a \emph{Koll\'ar component} over $x\in (X,\Delta)$ and $\pi\colon Y\to X$ the plt blowup of $E$.
\end{defn}

By adjunction, we may write 
\[
(K_Y+\Delta_Y+E)|_E = K_E+\Delta_E
\]
for some effective divisor $\Delta_E=\Diff_E(\Delta_Y)$ (called the different) on $E$, and $(E,\Delta_E)$ is a klt log Fano pair.

\begin{defn}[\cite{Z-mld^K-1}]
Let $x\in (X,\Delta)$ be a klt singularity. The minimal log discrepancy of Koll\'ar components, denoted $\mldk (x,X,\Delta)$, is the infimum of the log discrepancies $A_{X,\Delta}(E)$ as $E$ varies among all Koll\'ar components over $x\in (X,\Delta)$.
\end{defn}

If $x\in (X,\Delta)$ is a klt singularity, then we can write $\Delta=\sum_{i=1}^r a_i \Delta_i$ as a convex combination of $\bQ$-divisors $\Delta_i$ such that each $(X,\Delta_i)$ is klt. Let $r>0$ be an integer such that $r(K_X+\Delta_i)$ is Cartier for all $i$. Then as $A_{X,\Delta}(F)=\sum_{i=1}^r a_i A_{X,\Delta_i}(F)$, the possible values of log discrepancies $A_{X,\Delta}(F)$ belong to the discrete set $\left\{\left.\frac{1}{r}\sum_{i=1}^r a_i m_i\,\right|\, m_i\in\bN\right\}$. This implies that the infimum in the above definition is also a minimum.

The following result will be useful later. Recall that the log canonical threshold $\lct(X,\Delta;D)$ of an effective $\bR$-Cartier divisor $D$ with respect to a klt pair $(X,\Delta)$ is the largest number $t\ge 0$ such that $(X,\Delta+tD)$ is klt, and the $\alpha$-invariant of a log Fano pair $(X,\Delta)$ is defined as the infimum of the log canonical thresholds $\lct(X,\Delta;D)$ where $0\le D\sim_\bR -(K_X+\Delta)$. The log canonical threshold $\lct_x(X,\Delta;D)$ at a closed point $x\in X$ is defined analogously.

\begin{lem} \label{lem:alpha gives Izumi const}
Let $E$ be a Koll\'ar component over a klt singularity $x\in (X,\Delta)$. Then for any effective $\bR$-Cartier divisor $D$ on $X$, we have
\[
\lct_x(X,\Delta;D)\ge \min\{1,\alpha(E,\Delta_E)\} \cdot \frac{A_{X,\Delta}(E)}{\ord_E(D)}.
\]
\end{lem}

\begin{proof}
Let $\alpha = \min\{1,\alpha(E,\Delta_E)\}$, let $t=\frac{A_{X,\Delta}(E)}{\ord_E(D)}$ and let $\pi\colon Y\to X$ be the plt blowup of $E$. Then we have
\[
\pi^*(K_X+\Delta+tD) = K_Y+\Delta_Y+tD_Y+E,
\]
which gives $tD_Y|_E\sim_\bR -(K_Y+\Delta_Y+E)|_E \sim_\bR -(K_E+\Delta_E)$. By the definition of alpha invariants, the pair $(E,\Delta_E+\alpha tD_Y|_E)$ is log canonical, hence by inversion of adjunction \cite[Theorem 5.50]{KM98} we know that $(Y,\Delta_Y+ \alpha tD_Y+E)$ is lc around $E$. As $\alpha\le 1$, we also have
\[
\pi^*(K_X+\Delta+\alpha tD) = K_Y+\Delta_E+\alpha tD_Y+sE
\]
for some $s\le 1$, thus by the above discussion we deduce that $(Y,\Delta_Y+ \alpha tD_Y+sE)$ is sub-lc around $E$, and hence $(X,\Delta+\alpha tD)$ is lc at $x$. In other words, $\lct_x(X,\Delta;D)\ge \alpha t$ as desired.
\end{proof}

\subsection{Local volumes} \label{ss:local vol}

We next briefly recall the definition of the local volumes of klt singularities \cite{Li-normalized-volume}. Let $x\in (X,\Delta)$ be a klt singularity and let $n=\dim X$. The \emph{log discrepancy} function
\[
A_{X,\Delta}\colon \Val_{X,x}\to \bR \cup\{+\infty\},
\]
is defined as in \cite{JM-val-ideal-seq} and \cite[Theorem 3.1]{BdFFU-log-discrepancy}. It generalizes the usual log discrepancies of divisors; in particular, for divisorial valuations, i.e. valuations of the form $\lambda\cdot \ord_F$ where $\lambda>0$ and $F$ is a divisor over $X$, we have 
\[
A_{X,\Delta}(\lambda\cdot \ord_F)=\lambda\cdot A_{X,\Delta}(F).
\]
We denote by $\Val^*_{X,x}$ the set of valuations $v\in\Val_X$ with center $x$ and  $A_{X,\Delta}(v)<+\infty$. The \emph{volume} of a valuation $v\in\Val_{X,x}$ is defined as
\[
\vol(v)=\vol_{X,x}(v)=\limsup_{m\to\infty}\frac{\ell(\cO_{X,x}/\fa_m(v))}{m^n/n!}.
\] 

\begin{defn} \label{defn:local volume}
Let $x\in (X,\Delta)$ be an $n$-dimensional klt singularity. For any $v\in \Val^*_{X,x}$, we define the \emph{normalized volume} of $v$ as
\[
\hvol_{X,\Delta}(v):=A_{X,\Delta}(v)^n\cdot\vol_{X,x}(v).
\]
The \emph{local volume} of $x\in (X,\Delta)$ is defined as
\[
  \hvol(x,X,\Delta):=\inf_{v\in\Val^*_{X,x}} \hvol_{X,\Delta}(v).
\]
\end{defn}

By \cite[Theorem 1.2]{Li-normalized-volume}, the local volume of a klt singularity is always positive. We will frequently use the following properties of local volumes.

\begin{lem}[{\cite[Theorem 1.6]{LX-cubic-3fold}}] \label{lem:vol<=n^n}
Let $x\in(X,\Delta)$ be a klt singularity of dimension $n$. Then $\hvol(x,X,\Delta)\le n^n$.
\end{lem}

\begin{lem}[{\cite[Corollary 1.4]{XZ-minimizer-unique}}] \label{lem:index bound via volume}
Let $x\in(X,\Delta)$ be a klt singularity of dimension $n$ and let $D$ be a $\bQ$-Cartier Weil divisor on $X$. Then the Cartier index of $D$ is at most $\frac{n^n}{\hvol(x,X,\Delta)}$.
\end{lem}

\subsection{Log Fano cone singularities} \label{ss:log Fano cone}

In this subsection we recall the definition of log Fano cone singularities and their K-semistability. These notions originally appear in the study of Sasaki-Einstein metrics \cites{CS-Kss-Sasaki,CS-Sasaki-Einstein} and is further explored in works related to the Stable Degeneration Conjecture \cites{LX-stability-higher-rank,LWX-metric-tangent-cone}.

\begin{defn}
Let $X=\Spec(R)$ be a normal affine variety and $\bT=\bG_m^r$ ($r>0$) an algebraic torus. We say that a $\bT$-action on $X$ is {\it good} if it is effective and there is a unique closed point $x\in X$ that is in the orbit closure of any $\bT$-orbit. We call $x$ the vertex of the $\bT$-variety $X$, and call the corresponding singularity $x\in X$ a $\bT$-singularity.
\end{defn}

Let $N:=N(\bT)=\Hom(\bG_m, \bT)$ be the co-weight lattice and $M=N^*$ the weight lattice. We have a weight decomposition 
\[
R=\oplus_{\alpha\in M} R_\alpha,
\]
and the action being good implies that $R_0=\bk$ and every $R_\alpha$ is finite dimensional. For $f\in R$, we denote by $f_\alpha$ the corresponding component in the above weight decomposition.

\begin{defn}
A \emph{Reeb vector} on $X$ is a vector $\xi\in N_\bR$ such that $\langle \xi, \alpha \rangle>0$ for all $0\neq \alpha\in M$ with $R_{\alpha}\neq 0$. The set $\ft^+_{\bR}$ of Reeb vectors is called the Reeb cone.
\end{defn}

For any $\xi\in \ft^+_{\bR}$, we can define a valuation $\wt_\xi$ (called a toric valuation) by setting
\[
\wt_\xi (f):=\min\{\langle \xi, \alpha \rangle\mid \alpha\in M, f_\alpha\neq 0\}
\]
where $f\in R$. It is not hard to verify that $\wt_\xi\in \Val_{X,x}$.

\begin{defn}
A log Fano cone singularities is a klt singularities that admits a nontrivial good torus action. A polarized log Fano cone singularity $x\in (X,\Delta;\xi)$ consists of a log Fano cone singularity $x\in (X,\Delta)$ together with a Reeb vector $\xi$ (called a polarization). 
\end{defn}

By abuse of convention, a good $\bT$-action on a klt singularity $x\in (X,\Delta)$ means a good $\bT$-action on $X$ such that $x$ is the vertex and $\Delta$ is $\bT$-invariant. Using terminology from Sasakian geometry, we say a polarized log Fano cone $x\in (X,\Delta;\xi)$ is \emph{quasi-regular} if $\xi$ generates a $\bG_m$-action (i.e. $\xi$ is a real multiple of some element of $N$); otherwise, we say that $x\in (X,\Delta;\xi)$ is \emph{irregular}.

We will often use the following result to perturb an irregular polarization to a quasi-regular one.

\begin{lem} \label{lem:vol continuous in Reeb cone}
The function $\xi\mapsto \hvol_{X,\Delta}(\wt_\xi)$ defined on the Reeb cone is continuous and has a minimum.
\end{lem}

\begin{proof}
This follows from \cite[Theorem 2.15(3) and Proposition 2.39]{LX-stability-higher-rank}.
\end{proof}

\begin{defn}
We say a polarized log Fano cone singularity $x\in (X,\Delta;\xi)$ is \emph{K-semistable} if 
\[
\hvol(x,X,\Delta)=\hvol_{X,\Delta}(\wt_\xi).
\]
\end{defn}

This definition differs from the original ones from \cites{CS-Kss-Sasaki,CS-Sasaki-Einstein}, but they are equivalent by \cite[Theorem 2.34]{LX-stability-higher-rank}. For our purpose, the above definition is more convenient. Since the minimizer of the normalized volume function is unique up to rescaling \cite{XZ-minimizer-unique}, the polarization $\xi$ is essentially determined by the K-semistability condition and hence we often omit the polarization and simply say $x\in (X,\Delta)$ is K-semistable.

\begin{defn} \label{def:vol ratio}
The \emph{volume ratio} of a polarized log Fano cone singularity $x\in (X,\Delta;
\xi)$ is defined to be
\[
\Theta(X,\Delta;
\xi):=\frac{\hvol(x,X,\Delta)}{\hvol_{X,\Delta}(\wt_\xi)}.
\]
The \emph{volume ratio} of a log Fano cone singularity $x\in (X,\Delta)$ is defined to be
\[
\Theta(x,X,\Delta):=\sup_{\xi} \Theta(X,\Delta;\xi),
\]
where the supremum runs over all polarizations on $X$.
\end{defn}

By definition, $0<\Theta(X,\Delta;
\xi)\le 1$ and $\Theta(X,\Delta;
\xi)=1$ if and only if $x\in (X,\Delta;
\xi)$ is K-semistable. By Lemma \ref{lem:vol continuous in Reeb cone}, similar statement holds in the unpolarized case.

\subsection{Bounded family of singularities} \label{ss:bdd defn}

In this subsection we define boundedness of singularities and recall some properties of singularity invariants in bounded families. They will be useful in proving the easier direction of Theorem \ref{main:K-ss cone bdd}.

\begin{defn}
We call $B\subseteq (\cX,\cD)\to B$ an $\bR$-Gorenstein family of klt singularities (over a normal but possibly disconnected base $B$) if
\begin{enumerate}
    \item $\cX$ is flat over $B$, and $B\subseteq \cX$ is a section of the projection,
    \item For any closed point $b\in B$, $\cX_b$ is connected, normal and is not contained in $\Supp(\cD)$,
    \item $K_{\cX/B}+\cD$ is $\bR$-Cartier and $b\in (\cX_b,\cD_b)$ is a klt singularity for any $b\in B$.
\end{enumerate}
\end{defn}

\begin{lem} \label{lem:vol and mld^K bdd in family}
Let $B\subseteq (\cX,\cD)\to B$ be an $\bR$-Gorenstein family of klt singularities. Then there exist constants $\varepsilon,A>0$ such that 
\[
\hvol(b,\cX_b,\cD_b)\ge \varepsilon \quad \mathrm{and}\quad \mldk(b,\cX_b,\cD_b)\le A
\]
for all closed point $b\in B$.
\end{lem}

\begin{proof}
The uniform lower bound of the local volumes follows from their lower semi-continuity \cite{BL-vol-lsc} in families together with Noetherian induction. Note that although the main result of \cite{BL-uks-openness} is stated for families with $\bQ$-coefficients, it remains valid for real coefficients. This is because there exist $\bQ$-divisors $\cD'\le \cD$ such that $K_{\cX/B}+\cD'$ is $\bQ$-Cartier and the coefficients of $\cD-\cD'$ are arbitrarily small; in particular, $\hvol(b,\cX_b,\cD_b)\ge \hvol(b,\cX_b,\cD'_b)$ and their difference can be made arbitrarily small for any fixed $b\in B$. Thus the lower semicontinuity of $\hvol(b,\cX_b,\cD_b)$ in $b\in B$ follows from the lower semicontinuity of $\hvol(b,\cX_b,\cD'_b)$. Alternatively, the local volume lower bound follows from the constructibility of the volume function \cite[Theorem 1.3]{Xu-quasi-monomial}, which also holds for families with real coefficients by the discussions in Theorem \ref{thm:SDC}.

The uniform upper bound of $\mldk$ is a direct consequence of \cite[Theorem 2.34]{HLQ-vol-ACC}: after a (quasi-finite and surjective) base change, the family admits a flat family of Koll\'ar components (in the sense of \cite[Definition 2.32]{HLQ-vol-ACC}), whose log discrepancy is locally constant and hence uniformly bounded.
\end{proof}

We next define the boundedness of log Fano cone singularities. Note that our definition is a priori stronger than their boundedness as klt singularities (it's not clear to us whether they are equivalent).

\begin{defn} \label{defn:cone bdd}
We say that a set $\cS$ of polarized log Fano cone singularities is bounded if there exists finitely many $\bR$-Gorenstein families $B_i\subseteq (\cX_i,\cD_i)\to B_i$ of klt singularities, each with a fiberwise good $\bT_i$-action for some nontrivial algebraic torus $\bT_i$, such that every $x\in (X,\Delta;\xi)$ in $\cS$ is isomorphic to $b\in (\cX_{i,b},\cD_{i,b};\xi_b)$ for some $i$, some $b\in B_i$ and some $\xi_b\in N(\bT_i)_\bR$. We say $\cS$ is log bounded if we only ask $x\in (X;\xi)$ to be isomorphic to $b\in (\cX_{i,b};\xi_b)$ and that $\Supp(\Delta)\subseteq \Supp(\cD_{i,b})$ under this isomorphism.

We say that a set $\cC$ of log Fano cone singularities is bounded if it's the underlying set of singularities from a bounded set of polarized log Fano cone singularities.
\end{defn}

\begin{lem} \label{lem:density bdd in family}
Let $\cC$ be a bounded set of log Fano cone singularities. Then there exists some constant $\theta>0$ such that $\Theta(x,X,\Delta)\ge \theta$ for all $x\in(X,\Delta)$ in $\cC$.
\end{lem}

\begin{proof}
By definition, it suffices to show that for any $\bR$-Gorenstein family of log Fano cone singularities $B\subseteq (\cX,\cD)\to B$ with a fiberwise good torus $\bT$-action, we have a uniform positive lower bound of $\Theta(b,\cX_b,\cD_b)$. We may assume that $B$ is connected. Replacing $\cX$ by $\bT$-invariant affine subset, We may also assume that $\cX=\Spec(\cR)$ is affine. We have a weight decomposition $\cR=\oplus_{\alpha\in M} \cR_\alpha$ which reduces to the weight decomposition on the fibers. By flatness of $\cX\to B$, each $\cR_\alpha$ is flat over $B$. It follows that the Reeb cone $\ft^+_{\bR,b}\subseteq N_\bR$ is locally constant in $b\in B$, hence is constant as $B$ is connected. Choose any primitive $\xi\in N$ that lies in the (constant) Reeb cone. It suffices to show that $\Theta(b,\cX_b,\cD_b;\xi)$ is uniformly bounded from below. Since the $\bT$-action is good, the $\cR_\alpha$'s are also finite over $B$, hence locally free by flatness. From the defition of volume, this easily implies that $\vol(\wt_\xi)$ is constant as $b\in B$ varies. Since $\xi\in N$ comes from a $\bG_m$-action, there is a divisor $\cE$ over $\cX$ such that $\ord_{\cE}=\wt_\xi$ as valuations over $\cX$ and hence $A_{\cX_b,\cD_b}(\wt_\xi)=A_{\cX_b,\cD_b}(\ord_{\cE_b})=A_{\cX,\cD}(\cE)$ for general $b\in B$. By Noetherian induction, this implies that $A_{\cX_b,\cD_b}(\wt_\xi)$ is constructible in $b\in B$ and thus uniformly bounded from above. Finally, by Lemma \ref{lem:vol and mld^K bdd in family} we know that $\hvol(b,\cX_b,\cD_b)\ge \varepsilon$ for some constant $\varepsilon>0$. Putting these together we deduce from the definition that $\Theta(b,\cX_b,\cD_b;\xi)$ is uniformly bounded from below.
\end{proof}




\subsection{Real coefficients} \label{ss:real coef}

The following result will help us reduce many questions about pairs with real coefficients to ones with rational coefficients. We will often use it without explicit mention when we refer to results that are originally stated for $\bQ$-coefficients.

\begin{lem} \label{lem:R to Q}
Let $n\in \bN$ and let $I\subseteq [0,1]$ be a finite set. Then there exists a finite set $I'\subseteq [0,1]\cap\bQ$ depending only on $n$ and $I$ such that the following holds.

For any $n$-dimensional klt singularity $x\in (X,\Delta)$ with coefficients in $I$, there exists some effective $\bQ$-divisor $\Delta'\ge \Delta$ on $X$ with coefficients in $I'$, such that $\Supp(\Delta)=\Supp(\Delta')$, $x\in (X,\Delta')$ is klt, $\hvol(x,X,\Delta')\ge 2^{-n}\hvol(x,X,\Delta)$ and $\mldk(x,X,\Delta')\le 2\cdot \mldk(x,X,\Delta)$. If in addition $x\in (X,\Delta;\xi)$ is a log Fano cone singularity, then $\Theta(x,X,\Delta';\xi)\ge 4^{-n}\Theta(x,X,\Delta;\xi)$.
\end{lem}

\begin{proof}
This is essentially a consequence of \cite[Theorem 5.6]{HLS-epsilon-plt-blowup}. Without loss of generality we may assume $1\in I$. Write $I=\{a_1,\dots,a_m\}$ and let $\va=(a_1,\dots,a_m)\in \bR^m$. Let $V\subseteq \bR^m$ be the rational envelope of $\va$, i.e. the smallest affine subspace defined over $\bQ$ that contains $\va$. By \cite[Theorem 5.6]{HLS-epsilon-plt-blowup}, there exists an open neighbourhood $U$ of $\va\in V$ depending only on $n$ and $I$ such that for any $n$-dimensional singularity $x\in X$ and any Weil divisors $\Delta_1,\dots,\Delta_m\ge 0$, if $x\in (X,\sum_{i=1}^m a_i \Delta_i)$ is klt (resp. plt) then $x\in (X,\sum_{i=1}^m a'_i \Delta_i)$ is klt (resp. plt) for any $\va'=(a'_1,\dots,a'_m)\in U$. Note that the plt case implies that if $E$ is a Koll\'ar component over $x\in (X,\sum_{i=1}^m a_i \Delta_i)$ then it is also a Koll\'ar component over $x\in (X,\sum_{i=1}^m a'_i \Delta_i)$.

Choose some $\va'=(a'_1,\dots,a'_m)\in U\cap\bQ^m$ such that $2\va'-\va, 2\va-\va'\in U$. We claim that $I'=\{a'_1,\dots,a'_m\}$ satisfies the required conditions. Indeed, for any $n$-dimensional klt singularity $x\in (X,\Delta=\sum_{i=1}^m a_i \Delta_i)$ (where the $\Delta_i$'s are Weil divisors), if we set $\Delta'=\sum_{i=1}^m a'_i \Delta_i$ and let $\Delta_1=2\Delta'-\Delta$, then by our choice of $U$ and $\va'$ we know that $\Supp(\Delta')=\Supp(\Delta)$, $x\in (X,\Delta')$ and $x\in (X,\Delta_1)$ are both klt, and any Koll\'ar component over $x\in (X,\Delta)$ is also a Koll\'ar component over $x\in (X,\Delta')$. Moreover, since $\Delta'=\frac{1}{2}(\Delta+\Delta_1)$, we have $A_{X,\Delta'}(v)=\frac{1}{2}(A_{X,\Delta}(v)+A_{X,\Delta_1}(v))\ge \frac{1}{2}A_{X,\Delta}(v)$ for any valuation $v\in \Val_{X,x}^*$. Similarly using $2\va-\va'\in U$ we get $A_{X,\Delta'}(v)\le 2A_{X,\Delta}(v)$. In particular we see that $2^{-n}\hvol_{X,\Delta}(v)\le \hvol_{X,\Delta'}(v)\le 2^n\hvol_{X,\Delta}(v)$. These imply the inequalities about $\hvol$, $\mldk$ and the volume ratio (all by definition).
\end{proof}

We also note the Stable Degeneration Conjecture holds for pairs with real coefficients.

\begin{thm}[\cites{Blu-minimizer-exist,LX-stability-higher-rank,Xu-quasi-monomial,XZ-minimizer-unique,XZ-SDC}] \label{thm:SDC}
Every klt singularity $x\in (X,\Delta)$ has a special degeneration to a K-semistable log Fano cone singularity $x_0\in (X_0,\Delta_0)$ with $\hvol(x,X,\Delta)=\hvol(x_0,X_0,\Delta_0)$
\end{thm}

\begin{proof}
The arguments in \cites{Blu-minimizer-exist,LX-stability-higher-rank,XZ-minimizer-unique,XZ-SDC} extend directly to real coefficients, since they are not sensitive to the coefficients. The proof in \cite{Xu-quasi-monomial} uses the existence of monotonic $N$-complement which requires $\bQ$-coefficients, but the main results in \emph{loc. cit.} can be extended to $\bR$-coefficients by \cite[Theorems 3.3 and 3.4]{HLQ-vol-ACC}. Thus \cite[Theorem 1.2]{XZ-SDC} holds for $\bR$-pairs. In particular, the minimizer $v$ of the normalized volume function $\hvol_{X,\Delta}$ induces a special degeneration of $x\in (X=\Spec(R),\Delta)$ to a K-semistable log Fano cone $x_0\in (X_0=\Spec(\gr_v R),\Delta_0;
\xi_v)$. By \cite[Lemma 2.58]{LX-stability-higher-rank}, we have
\[
\hvol_{X,\Delta}(v)=\hvol_{X_0,\Delta_0}(\wt_{\xi_v}).
\]
Since $\hvol(x,X,\Delta)=\hvol_{X,\Delta}(v)$ (as $v$ is the minimizer of $\hvol_{X,\Delta}$) and $\hvol(x_0,X_0,\Delta_0)=\hvol_{X_0,\Delta_0}(\wt_{\xi_v})$ (as $x_0\in (X_0,\Delta_0;\xi_v)$ is K-semistable), we see that the degeneration preserves the local volume.
\end{proof}

\subsection{Results from Part I} \label{ss:part I}

In this subsection we collect slight modification of several results from \cite{Z-mld^K-1} that we need in later proofs. For the first result, recall that we say a singularity $x\in (X,\Delta)$ is of klt type if there exists some effective $\bR$-divisor $D$ on $X$ such that $x\in (X,\Delta+D)$ is klt.

\begin{lem} \label{lem:klt type threshold}
Let $n\in\bN$ and let $\varepsilon>0$. Then there exists some constant $c=c(n,\varepsilon)>0$ such that for any $n$-dimensional klt singularity $x\in (X,\Delta)$ with $\hvol(x,X,\Delta)\ge \varepsilon$, we have that $x\in (X,(1+c)\Delta)$ is of klt type.
\end{lem}

\begin{proof}
Since $x\in (X,\Delta)$ is klt, by \cite{BCHM} (\emph{cf.} \cite[Lemma 4.7]{Z-mld^K-1}) there exists a small birational morphism $\pi\colon Y\to X$ such that $K_Y$ is $\bQ$-Cartier and relatively ample. Since $\pi$ is small, we have $K_Y+\Delta_Y=\pi^*(K_X+\Delta)$, hence by \cite[Lemma 2.9(2)]{LX-cubic-3fold} (\emph{cf.} the proof of \cite[Lemma 2.10]{Z-mld^K-1}) we have 
\[
\hvol(y,Y,\Delta_Y)\ge \hvol(x,X,\Delta)\ge \varepsilon
\]
for all $y\in \pi^{-1}(x)$. By \cite[Theorem 3.1]{Z-mld^K-1}, there exists some constant $c=c(n,\varepsilon)>0$ such that the pair $(Y,(1+c)\Delta_Y)$ is klt for all $y\in \pi^{-1}(x)$. Note that $-(K_Y+(1+c)\Delta_Y)\sim_{\bQ,\pi} cK_Y$, hence $\pi$ is also the ample model of $-(K_X+(1+c)\Delta)$. By \cite[Lemma 2.4]{Z-direct-summand}, this implies that $x\in (X,(1+c)\Delta)$ is of klt type.
\end{proof}

The second result is extracted from the proof of \cite[Theorem 4.1]{Z-mld^K-1}.

\begin{lem} \label{lem:Cartier index on plt blowup}
Let $n\in\bN$, let $\varepsilon,A>0$, and let $I\subseteq [0,1]\cap \bQ$ be a finite set. Then there exists some integer $N=N(n,\varepsilon,A,I)>0$ such that for any $n$-dimensional klt singularity $x\in (X,\Delta)$ with coefficients in $I$ and $\hvol(x,X,\Delta)\ge \varepsilon$, and for any Koll\'ar component $E$ over $x\in (X,\Delta)$ with $A_{X,\Delta}(E)\le A$, we have that $NE$ and $N(K_Y+\Delta_Y+E)$ are Cartier on the plt blowup $Y\to X$ of $E$.
\end{lem}

\begin{proof}
By \cite[(4.1)]{Z-mld^K-1}, the local volumes $\hvol(y,Y,\Delta_Y)$ (where $y\in E$) are bounded from below by some constants that only relies on the given constants $n,\varepsilon,A$. Thus by Lemma \ref{lem:index bound via volume} we get the desired Cartier index bound.
\end{proof}




\section{Boundedness} \label{s:proof}

In this section, we prove the following statement on boundedness of log Fano cone singularities, which will imply the main results of this paper.

\begin{thm} \label{thm:polarized cone bdd}
Let $n\in \bN$ and let $I\subseteq [0,1]$ be a finite set. Let $\varepsilon,\theta,A>0$. Let $\cS$ be the set of $n$-dimensional polarized log Fano cone singularities $x\in (X,\Delta;\xi)$ with coefficients in $I$ such that
\[
\hvol(x,X,\Delta)\ge \varepsilon,\quad 
\Theta(X,\Delta;\xi)\ge \theta,\quad\mathrm{and}\quad
\mldk(x,X,\Delta)\le A.
\]
Then $\cS$ is bounded.
\end{thm}

Recall that $\Theta(X,\Delta;\xi)$ is the volume ratio of the log Fano cone singularity (Definition \ref{def:vol ratio}). We refer to Sections \ref{ss:kc}--\ref{ss:bdd defn} for the other relevant definitions.

\subsection{Orbifold cones} \label{ss:orbifold}

Since the volume ratio is continuous in the Reeb vector $\xi$ (Lemma \ref{lem:vol continuous in Reeb cone}), by perturbing the polarization, we see that it suffices to prove Theorem \ref{thm:polarized cone bdd} when $\cS$ consists of quasi-regular log Fano cones. Every quasi-regular log Fano cone singularity has a natural affine orbifold cone structure induced by the polarization. The proof of Theorem \ref{thm:polarized cone bdd} relies on the associated projective orbifold cone construction. In this subsection, we first fix some notation and recall some basic properties of orbifold cones from \cite{Kol-Seifert-bundle}.

\begin{defn}
Let $V$ be a normal projective variety. Let $L$ be an ample $\bQ$-Cartier $\bQ$-divisor on $V$. 
\begin{enumerate}
    \item The \emph{affine orbifold cone} $C_a(V,L)$ is defined as
    \[
    C_a(V,L):=\Spec\bigoplus_{m=0}^\infty H^0(V,\cO_{V}(\lfloor mL \rfloor)).
    \]
    \item The \emph{projective orbifold cone} $C_p(V,L)$ is defined as
    \[
    C_p(V,L):=\Proj\bigoplus_{m=0}^\infty\bigoplus_{i=0}^\infty H^0(V,\cO_{V}(\lfloor mL \rfloor)\cdot s^i,
    \]
    where the grading of $H^0(V,\cO_{V}(\lfloor mL \rfloor))$ and $s$ are $m$ and $1$, respectively.
\end{enumerate}
\end{defn}

For ease of notation, denote $X=C_a(V,L)$, $\oX=C_p(V,L)$ and let $x\in X$ be the vertex of the orbifold cone. On the projective orbifold cone, we also denote by $V_\infty$ the divisor at infinity, i.e., the divisor corresponding to $(s=0)$. We have $X\cong \oX\setminus V_\infty$. Note that $X\setminus \{x\}$ is a Seifert $\bG_m$-bundle over $V$ (in the sense of \cite{Kol-Seifert-bundle}). Thus for any effective $\bR$-divisor $\Delta_V$ on $V$, we can define the affine (resp. projective) orbifold cone over the polarized pair $(V,\Delta_V;L)$ as the pair $(X,\Delta)$ (resp. $(\oX,\oDelta)$), where  $\Delta$ (resp. $\oDelta$) is the closure of the pullback of $\Delta_V$ to $X\setminus \{x\}$ (since the projection $X\setminus \{x\}\to V$ is equidimensional, the pullback of a Weil divisor is well-defined). Every quasi-regular polarized log Fano cone singularity $x\in (X,\Delta;\xi)$ arises in this way, so we can talk about its associated projective orbifold cone $(\oX,\oDelta)$.

Alternatively, the Seifert $\bG_m$-bundle $X\setminus \{x\}$ can be compactified by adding $V_\infty$ and the missing zero section $V_0$; the resulting space $\oY$ is also the orbifold blowup of $\oX$ at $x$. Similarly, let $Y$ be the orbifold blowup of $X$ at $x$ (i.e. $Y=\oY\setminus V_\infty$). Write the fractional part of $L$ as $\{L\}=\sum_{i=1}^k \frac{a_i}{b_i}D_i$ for some prime divisors $D_i$ on $V$ and $0<a_i<b_i$ coprime integers. We denote $\Delta_L:=\sum_{i=1}^k\frac{b_i-1}{b_i}D_i$. Then the pairs $(V_0,\Diff_{V_0}(0))$ and $(V_\infty, \Diff_{V_\infty}(0))$ obtained by taking adjunction over some big open set of $V$ from the pair $(\oX, V_0+V_\infty)$ are both isomorphic to $(V,\Delta_L)$ (this is a local computation, see \cite[Section 4]{Kol-Seifert-bundle}).

Under these notation, we have the following well known result \cite{Kol-Seifert-bundle} (\emph{cf.} \cite[Section 3.1]{Kol13} for analogous statement for usual cones).

\begin{lem} \label{lem:orbifold cone property}
The following conditions are equivalent:    
\begin{enumerate}
    \item $x\in (X,\Delta)$ is klt;
    \item $(\oX,\oDelta+V_\infty)$ is plt;
    \item $(V,\Delta_V+\Delta_L)$ is a log Fano pair, and $-(K_V+\Delta_V+\Delta_L)\sim_\bR rL$ for some $r>0$.
\end{enumerate}
Moreover, when the above conditions are satisfied, we have $K_{\oX}+\oDelta\sim_{\bR}-(1+r)V_\infty$ and $A_{X,\Delta}(V_0)=r$.
\end{lem}

The next result is a key observation in the proof of Theorem \ref{thm:polarized cone bdd}. It expresses the normalized volume of a polarized log Fano cone singularities as the global volume of the associated projective orbifold cone.

\begin{lem} \label{lem:Sasaki vol as global vol}
Assume that $x\in (X,\Delta)$ is klt. Then under the notation of Lemma \ref{lem:orbifold cone property} we have
\[
\hvol_{X,\Delta}(\ord_{V_0})=\vol(-(K_{\oX}+\oDelta+V_{\infty}))=r\cdot \vol(-(K_V+\Delta_V+\Delta_L)).
\]
\end{lem}

\begin{proof}
This follows from a direct calculation. First we have $A_{X,\Delta}(V_0)=r$ by Lemma \ref{lem:orbifold cone property}. We also have $-V_0|_{V_0}\cong L$, hence 
\[
\hvol_{X,\Delta}(\ord_{V_0})=A_{X,\Delta}(V_0)^n\cdot\vol_{X,\Delta}(\ord_{V_0})= r^n\vol(-V_0|_{V_0})=r^n L^{n-1}
\]
where $n=\dim X$. On the other hand, by Lemma \ref{lem:orbifold cone property} we get $-(K_{\oX}+\oDelta+V_{\infty})\sim_\bR rV_\infty$; recall also that $-(K_{\oX}+\oDelta+V_{\infty})|_{V_\infty}\sim_\bR -(K_V+\Delta_V+\Delta_L)\sim_\bQ rL$ (where we identify $V_\infty$ with $V$), thus
\begin{align*}
    \vol(-(K_{\oX}+\oDelta+V_{\infty})) & = \left( -(K_{\oX}+\oDelta+V_{\infty})\right)^{n-1}\cdot (rV_\infty) \\
    & = r\cdot \vol(-(K_V+\Delta_V+\Delta_L)) = r\cdot (rL)^{n-1} = r^n L^{n-1},
\end{align*}
which proves the desired equality.
\end{proof}

We also need a slight generalization of Lemma \ref{lem:orbifold cone property}.

\begin{lem} \label{lem:cone klt type}
Notation as before. Assume that $x\in (X,\Delta)$ is of klt type. Then $-(K_V+\Delta_V+\Delta_L)$ is big.
\end{lem}

\begin{proof}
By assumption, there exists some effective $\bR$-divisor $D$ such that $x\in (X,\Delta+D)$ is klt (note that $D$ is not necessarily invariant under the $\bG_m$-action). In particular, $A_{X,\Delta+D}(V_0)>0$ and we have
\[
K_Y+\Delta_Y+D_Y+V_0\sim_\bR A_{X,\Delta+D}(V_0)\cdot V_0
\]
over $X$. It follows that $-(K_Y+\Delta_Y+D_Y+V_0)|_{V_0}\sim_\bR A_{X,\Delta+D}(V_0)\cdot L$ is ample. By adjunction, this implies that 
\[
-(K_V+\Delta_V+\Delta_L)=-(K_Y+\Delta_Y+D_Y+V_0)|_{V_0}+D_Y|_{V_0}
\]
is the sum of an ample divisor and an effective divisor, hence is big.
\end{proof}

\subsection{Effective birationality} \label{ss:bir bdd}

Given Lemma \ref{lem:Sasaki vol as global vol}, a na\"ive idea to prove Theorem \ref{thm:polarized cone bdd} is to associate to each log Fano cone singularity a projective orbifold cone and show that the corresponding set of projective orbifold cones is bounded. In general, this is too much to hope for, as the projective orbifold cone depends on the (auxiliary) choice of a quasi-regular polarization. The following example shows that even for a fixed singularity we can get an unbounded family of projective orbifold cones by choosing different polarizations.

\begin{expl}
Let $a_1,\dots,a_n\in\bN$ be pairwise coprime positive integers. Then $\xi=(a_1,\dots,a_n)\in \bN^n$ gives a polarization of the Fano cone singularity $0\in\bA^n$; it generates the $\bG_m$-action with weights $a_1,\dots,a_n$ on the coordinates. This endows $\bA^n$ with an affine orbifold cone structure $C_a(V,L)$ where $V=\bP(a_1,\dots,a_n)$ and $L=\cO_V(1)$. The associated projective orbifold cone is $\oX=\bP(1,a_1,\dots,a_n)$, which is clearly unbounded as the weights $a_i$'s vary. By choosing appropriate weights, we can even ensure that the normalized volume $\hvol(\wt_\xi)=\frac{(a_1+\dots+a_n)^n}{a_1\dots a_n}$ is fixed.
\end{expl}

Nonetheless, we observe that in the above example the projective orbifold cones we get satisfy the following interesting property: the linear system $|-K_{\oX}|$ always defines a birational map that is an embedding at the vertex $[1:0:\dots:0]$. In fact, if $[s:x_1:\dots:x_n]$ are the weighted homogeneous coordinates of $\oX$, then for every $i\in \{1,\dots,n\}$ there exists some $k_i\in \bN$ such that $s^{k_i} x_i\in H^0(-K_{\oX})$ (this is possible because $s$ has weight $1$); it is not hard to see that the sub-linear system spanned by $s^{k_i} x_i$ ($i=0,\dots,n$) is base point free and restricts to an embedding on the affine chart $\bA^n=\oX\setminus (s=0)$.

This motivates us to raise the following question.

\begin{que}
Let $n\in\bN$ and $\varepsilon>0$. Let $I\subseteq [0,1]\cap \bQ$ be a finite set. Is there some integer $m=m(n,\varepsilon,I)>0$ such that for any $n$-dimensional quasi-regular polarized log Fano cone $x\in (X,\Delta;\xi)$ with $\hvol(x,X,\Delta)\ge \varepsilon$, we have that $|-m(K_{\oX}+\oDelta+V_\infty)|$ defines a birational map that restricts to an embedding on $X$?
\end{que}

As before, $(\oX,\oDelta)$ denotes the associated projective orbifold cone, and $V_\infty=\oX\setminus X$ is the divisor at infinity.

The technical core of the proof of Theorem \ref{thm:polarized cone bdd} is the answer to this question after imposing upper bounds on the minimal log discrepancies of Koll\'ar components.

\begin{prop} \label{prop:bir bdd proj cone}
Let $n\in\bN$ and $\varepsilon,A>0$. Let $I\subseteq [0,1]\cap \bQ$ be a finite set. Then there is some integer $m=m(n,\varepsilon,A,I)>0$ such that for any $n$-dimensional quasi-regular polarized log Fano cone $x\in (X,\Delta;\xi)$ with $\hvol(x,X,\Delta)\ge \varepsilon$ and $\mldk(x,X,\Delta)\le A$, we have that $|-m(K_{\oX}+\oDelta+V_\infty)|$ defines a birational map that restricts to an embedding on $X$.
\end{prop}

We remark that the integer $m$ is independent of the polarization $\xi$.

\begin{proof}
We start with some reductions. Let $H:=-(K_{\oX}+\oDelta+V_\infty)$. We view $X$ as an affine orbifold cone $C_a(V,L)$ and keep the notation (e.g. $\Delta_V,\Delta_L$, etc) from Section \ref{ss:orbifold}. First note that since the coefficients of $\oDelta$ belong to a fixed finite set $I$ of rational numbers, we can choose some $m_0$ depending only on $I$ such that $m_0 H$ has integer coefficients. Since $\hvol(x,X,\Delta)\ge \varepsilon$, by Lemma \ref{lem:index bound via volume} we know that the Cartier index of $m_0 H$ at $x$ is bounded from above by $\frac{n^n}{\varepsilon}$. Thus replacing $m_0$ by a sufficiently large fixed multiple (e.g. $\lfloor\frac{n^n}{\varepsilon}\rfloor !$) we may further assume that $m_0 H$ is Cartier in a neighbourhood of $x$. By $\bG_m$-translation, this implies that $m_0 H$ is Cartier on $X$.

The next step is to produce enough sections (of some multiple of $m_0 H$) that do not vanish at $x$. We do this by creating isolated non-klt centers at $x$ and apply Nadel vanishing. Let $\varphi\colon\oY\to \oX$ be the orbifold blowup of the vertex $x$, and let $\Delta_{\oY}$ be the strict transform of $\oDelta$. Let $V_0$ be the exceptional divisor as before. Recall that $\oY$ has an orbifold $\bP^1$-bundle structure $\pi\colon \oY\to V$. Then we have (see \cite[Proposition 40 and Corollary 41]{Kol-Seifert-bundle})
\[
-(K_{\oY}+\Delta_{\oY}+V_0+V_\infty)\sim_\bQ -\pi^*(K_V+\Delta_V+\Delta_L).
\]
The left hand side is $\sim_\bQ \varphi^*H-aV_0$ where $a=A_{X,\Delta}(V_0)>0$, while the right hand side is semiample since $-(K_V+\Delta_V+\Delta_L)$ is ample on $V$ by Lemma \ref{lem:orbifold cone property}. From here we deduce that $\varphi^*H-aV_0$ is semiample. For any positive integer $m$, if we take $D$ to be a general member of the $\bQ$-linear system $\varphi_*|\varphi^*H-aV_0|_\bQ$, then $D\sim_\bQ H$ and $(\oX,\oDelta+mD)$ is klt away from $x$ (since this is the only base point of the $\bQ$-linear system). In particular, the multiplier ideal $\cJ(\oX,\oDelta+mD)$ is co-supported at $x$. Furthermore, we have
\[
\ord_{V_0}\cJ(\oX,\oDelta+mD)>m\cdot \ord_{V_0}(D)-A_{\oX,\oDelta}(V_0) = (m-1)a
\]
by the definition of multiplier ideals. Thus we obtain 
\begin{equation} \label{eq:multiplier contained in valuation ideal}
    \cJ(\oX,\oDelta+mD)\subseteq \fa_{(m-1)a},
\end{equation}
where $\fa_{(m-1)a}:=\fa_{(m-1)a}(\ord_{V_0})$ denotes the valuation ideals. Recall that $m_0 H$ is Cartier at $x$, hence as
\[
mH-(K_{\oX}+\oDelta+mD)\sim_\bQ -(K_{\oX}+\oDelta)
\]
is ample, by the following Lemma \ref{lem:Nadel vanishing for Weil div} (Nadel vanishing for Weil divisor), we have
\[
H^1(\oX, \cJ(\oX,\oDelta+mD)\otimes \cO_{\oX}(mH))=0
\]
for all $m\in\bN$ that's divisible by $m_0$. From the long exact sequence we then deduce that the natural map
\[
H^0(\oX,\cO_{\oX}(mH))\to H^0(\oX,\cO_{\oX}/\cJ(\oX,\oDelta+mD))
\]
is surjective. Combined with \eqref{eq:multiplier contained in valuation ideal}, we get a surjection
\begin{equation} \label{eq:separate jets from valuation ideal}
    H^0(\oX,\cO_{\oX}(mH))\to H^0(\oX,\cO_{\oX}/\fa_{(m-1)a})
\end{equation}
for all $m\in\bN$ that's divisible by $m_0$. 

At this point, we have produced sections of $mH$ that separate the jets in $\cO_{\oX}/\fa_{(m-1)a}$. Our goal is to make $|mH|$ into a birational map that restricts to an embedding on $X$. Clearly, a necessary condition is that $|mH|$ separates tangent directions at $x$. Let us first show that this latter condition can be achieved. In view of \eqref{eq:separate jets from valuation ideal}, it suffices to show that 
\begin{equation} \label{eq:val ideal in m^2}
    \fa_{(m-1)a}\subseteq \fm_x^2,
\end{equation}
which becomes a local question. If $X$ is fixed, this is certainly true for $m\gg 0$, so the main point is to make sure that the bound on $m$ depend only on the given data $n,\varepsilon,A$ rather than on $X$. To this end, note that one way to interpret the opposite condition $\fa_{(m-1)a}\not\subseteq \fm_x^2$ is that there exists some $f\in \fm_x\setminus \fm_x^2$ such that $\ord_{V_0}(f)\ge (m-1)a$; in particular, $\lct(f):=\lct(X,\Delta;(f=0))\le \frac{1}{m-1}$. This suggests to us that we should try to analyze the Izumi type constant
\begin{equation} \label{eq:min of lct}
    \inf\{\lct(f)\mid f\in \fm_x\setminus \fm_x^2\}
\end{equation}
for the singularities in question.

Let $E$ be a Koll\'ar component over $x\in (X,\Delta)$ such that $A_{X,\Delta}(E)\le A$, whose existence is guaranteed by our assumption that $\mldk(x,X,\Delta)\le A$. Let $f\colon Z\to X$ be the plt blowup that extract $E$ and let $\Delta_E=\Diff_E(\Delta_Z)$ be the different. By Lemma \ref{lem:alpha gives Izumi const}, we have
\begin{equation} \label{eq:lct bound Izumi type}
    \lct(f)\ge \min\{1,\alpha(E,\Delta_E)\}\cdot \frac{A_{X,\Delta}(E)}{\ord_E(f)}.
\end{equation}
In order to give uniform estimate of \eqref{eq:min of lct} using \eqref{eq:lct bound Izumi type}, we need further information about the log discrepancy $A_{X,\Delta}(E)$, the $\alpha$-invariant $\alpha(E,\Delta_E)$, and the value of 
\[
d_E:=\sup\{\ord_E(f)\mid f\in \fm_x\setminus \fm_x^2\}
\]
(the last one can be thought of as measuring how well the valuation ideals of $\ord_E$ approximate the first two powers of the maximal ideal $\fm_x$).



For the log discrepancy, recall from the previous discussion that $m_0(K_X+\Delta)$ is Cartier. Since $(X,\Delta)$ is klt, this implies that $A_{X,\Delta}(E)\ge \frac{1}{m_0}$. For the $\alpha$-invariant, we know by Lemma \ref{lem:Cartier index on plt blowup} that there exists some positive integer $N_1=N_1(n,\varepsilon,A,I)$ divisible by $m_0$ such that $N_1 E$ and $N_1(K_Z+\Delta_Z+E)$ are Cartier. By adjunction, we see that $N_1(K_E+\Delta_E)$ is also Cartier. By \cite[Corollary 1.8]{HMX-ACC}, the log Fano pair $(E,\Delta_E)$ belongs to a bounded family (which relies only on $N_1$). This implies that $\alpha(E,\Delta_E)$ is uniformly bounded below.
Finally, to give an upper bound for $d_E$, let $L=-E|_E$ be the (ample) $\bQ$-divisor defined as in \cite[Definition A.4]{HLS-epsilon-plt-blowup} and let
\[
\fb_k:=\fa_k(\ord_E)= f_*\cO_Z(-kE)
\]
where $k=1,2,\dots$. Then $N_1 L$ is Cartier as $N_1 E$ is Cartier. Since $-(K_E+\Delta_E)\sim_\bQ A_{X,\Delta}(E)\cdot L$ and $A_{X,\Delta}(E)\ge \frac{1}{m_0}$, we see that the coefficient and degree of $L$ are bounded, hence the triple $(E,\Delta,L)$ belongs to a bounded family. Thus there exists an integer $N=N(n,\varepsilon,A,I)>0$ such that the section ring $\oplus_{k\in\bN} H^0(E,\lfloor kL \rfloor)$ is generated in degree $\le N$. We claim that
\begin{equation} \label{eq:b_(N+1) contained in m^2}
    \fb_{N+1}\subseteq \fm_x^2.
\end{equation}
Taking this for granted, it follows immediately that $d_E\le N$. Putting these information together we deduce from \eqref{eq:lct bound Izumi type} that \eqref{eq:min of lct} is bounded below by some positive constant that only rely on $n,\varepsilon,A,I$. From the discussion right above \eqref{eq:min of lct} we also know that \eqref{eq:min of lct} is bounded from above by $\frac{1}{m-1}$ where $m$ is the largest integer such that $\fa_{(m-1)a}\not\subseteq \fm_x^2$. Thus  we see that there is some fixed $m$ depending only on $n,\varepsilon,A,I$ such that $mH$ is Cartier  and \eqref{eq:val ideal in m^2} holds. Combined with \eqref{eq:separate jets from valuation ideal}, we deduce that
\begin{equation} \label{eq:separate tangent}
    H^0(\oX,\cO_{\oX}(mH))\to H^0(\oX,\cO_{\oX}/\fm_x^2)
\end{equation}
is surjective, i.e. $|mH|$ separates tangent direction at $x$. 

Before we proceed to show that $|mH|$ also defines a birational map, let us finish the proof of the claim \eqref{eq:b_(N+1) contained in m^2}. In fact, we shall prove by descending induction that $\fb_k\subseteq \fm_x^2$ for all $k\ge N+1$. This is clear when $k\gg 0$. Suppose that $\fb_{k+1}\subseteq \fm_x^2$ and $k\ge N+1$, then since $\fb_{k+1}\subseteq \fb_k$ for all $k$ and since $\oplus_{k\in\bN}\fb_k/\fb_{k+1}\cong \oplus_{k\in\bN} H^0(E,\lfloor kL \rfloor)$ (see \cite[Section 2.4]{LX-stability-kc} or \cite[Proposition 2.10]{LZ-Tian-sharp}) is generated in degree $\le N$, we see that 
\[
\fb_k/\fb_{k+1}\subseteq \sum_{i=1}^N (\fb_i/\fb_{i+1})\cdot (\fb_{k-i}/\fb_{k-i+1})\subseteq \fb_1^2/\fb_{k+1};
\]
in other words, $\fb_k\subseteq \fb_{k+1}+\fb_1^2$. As $\fb_{k+1}\subseteq \fm_x^2$ by induction hypothesis and clearly $\fb_1\subseteq \fm_x$, we obtain $\fb_k\subseteq \fm_x^2$ as desired.

We are finally in a position to show that $|mH|$ defines a birational map that restricts to an embedding on $X$. By \eqref{eq:separate tangent}, we already know that in a neighbourhood of $x$, the linear system $|mH|$ is base point free and the induced map is unramified. Since the base locus of $|mH|$ is closed and invariant under the $\bG_m$-action (coming from the orbifold cone structure), we see that $|mH|$ has no base point in $X$. Similarly, as the ramification locus of the induced map $\varphi$ on $X$ is a closed $\bG_m$-invariant subset, we see that $\varphi$ is unramified on $X$ and thus it is quasi-finite.  This implies that $\varphi(x')\neq \varphi(x)$ for all $x'\neq x\in X$, otherwise by $\bG_m$-translation we deduce that $\varphi$ contracts the closed orbit $\overline{\bG_m\cdot x'}$. In particular, $\varphi|_X^{-1}(\varphi(x))$ is supported at $x$; as $\varphi$ is also unramified, the scheme-theoretic pre-image $\varphi|_X^{-1}(\varphi(x))$ equals $\{x\}$. By upper semi-continuity and the $\bG_m$-action, this implies that $\varphi|_X^{-1}(\varphi(x'))$ has length $1$ and thus consists of a single point for all $x'\in X$. It follows that $\varphi|_X$ is an embedding on $X$ and we finish the proof.
\end{proof}

We have used the following vanishing result in the above proof. This should be well-known to experts, but we are unable to find a suitable reference.

\begin{lem} \label{lem:Nadel vanishing for Weil div}
Let $(X,\Delta)$ be a pair such that $K_X+\Delta$ is $\bQ$-Cartier. Let $L$ be a $\bQ$-Cartier Weil divisor such that $L-(K_X+\Delta)$ is nef and big. Assume that $(X,\Delta)$ is klt along the non-Cartier locus of $L$. Then
\[
H^i(X,\cJ(X,\Delta)\otimes \cO_X(L))=0
\]
for all $i>0$.
\end{lem}

\begin{proof}
If $L$ is a line bundle this is just the usual Nadel vanishing; in the general case we follow the proof of Nadel vanishing. Let $f\colon Y\to X$ be a log resolution. We may write
\begin{align*}
    K_Y+\Delta_Y &=f^*(K_X+\Delta)+\sum a_i E_i ,\\
    f^*L &= \lfloor f^*L \rfloor + \sum b_i E_i ,
\end{align*}
where the $E_i$'s are the exceptional divisors, and $0\le b_i<1$. Let
\begin{align*}
    L' &= \lfloor f^*L \rfloor -\lfloor \Delta_Y \rfloor + \sum \lceil a_i+b_i \rceil E_i ,\\
    D &= \{\Delta_Y\}+\sum\{-a_i-b_i\}E_i.
\end{align*}
Then it's easy to check that $(Y,D)$ is klt (i.e. $\{D\}=0$ since $Y$ is a log resolution) and $L'-(K_Y+D)\sim_\bQ f^*(L-K_X-\Delta)$ is nef and big. By Kawamata-Viehweg vanishing we have $H^i(Y,\cO_Y(L'))=0$ and $R^if_*\cO_Y(L')=0$ for all $i>0$. It follows that $Rf_*\cO_Y(L')=f_*\cO_Y(L')$ and hence
\[
H^i(X,f_*\cO_Y(L'))=H^i(X,Rf_*\cO_Y(L'))=H^i(Y,\cO_Y(L'))=0.
\]
It remains to show that
\begin{equation} \label{eq:pushforward = multiplier ideal}
    f_*\cO_Y(L')=\cJ(X,\Delta)\otimes \cO_X(L).
\end{equation}
This is a local question, so it suffices to check the equality at any $x\in X$. If $(X,\Delta)$ is klt at $x$, then locally $\cJ(X,\Delta)=\cO_{X}$, $\lfloor \Delta_Y \rfloor = 0$, and $a_i>-1$ which implies $\lceil a_i+b_i \rceil\ge 0$. It follows that
\[
f_*\cO_Y(L')=\cO_X(L)=\cJ(X,\Delta)\otimes \cO_X(L).
\]
If $(X,\Delta)$ is not klt at $x$, then $L$ is Cartier around $x$, which gives $b_i=0$ for every $E_i$ whose image contains $x$, thus locally
\[
f_*\cO_Y(L')=f_*\cO_Y(f^*L+\lceil -\Delta_Y+\sum a_i E_i\rceil)=\cJ(X,\Delta)\otimes \cO_X(L)
\]
by projection formula and the definition of the multiplier ideal. Thus we see that \eqref{eq:pushforward = multiplier ideal} always holds. This completes the proof.
\end{proof}

\subsection{Proof of Theorem \ref{thm:polarized cone bdd}}

We are now in a position to prove Theorem \ref{thm:polarized cone bdd}. 

\begin{proof}
First assume that $I\subseteq [0,1]\cap \bQ$. Let $x\in (X,\Delta;\xi)$ be a polarized log Fano cone singularity in $\cS$ and let $\bT$ be the torus generated by $\xi$. By Lemma \ref{lem:vol continuous in Reeb cone}, there exists some quasi-regular polarization $\xi'$ that is sufficiently close to $\xi$ in the Reeb cone such that $\Theta(X,\Delta;\xi')\ge \frac{\theta}{2}$. Using $\xi'$, we may realize $(X,\Delta)$ as an orbifold cone over some polarized pair $(V,\Delta_V;L)$. Moreover, if $V_0$ is the exceptional divisor of the orbifold vertex blowup as in Section \ref{ss:orbifold}, then $\ord_{V_0}$ is proportional to $\wt_{\xi'}$. By assumption and Lemma \ref{lem:vol<=n^n}, we obtain 
\begin{equation} \label{eq:Sasaki vol upper bdd}
    \hvol_{X,\Delta}(\ord_{V_0})=\hvol_{X,\Delta}(\wt_{\xi'})\le 2\theta^{-1}\hvol(x,X,\Delta)\le 2\theta^{-1}n^n.
\end{equation}

Let $(\oX,\oDelta)$ be the projective orbifold cone over $(V,\Delta_V;L)$ and $V_\infty$ the divisor at infinity. By Proposition \ref{prop:bir bdd proj cone}, there exists some integer $m$ depending only on $n,\varepsilon,A,I$ such that $|-m(K_{\oX}+\oDelta+V_\infty)|$ defines a birational map $\varphi$ that restricts to an embedding on $X$. By \eqref{eq:Sasaki vol upper bdd} and Lemma \ref{lem:Sasaki vol as global vol}, we see that
\[
\vol(-m(K_{\oX}+\oDelta+V_\infty))=m^n\cdot \hvol_{X,\Delta}(\ord_{V_0}) \le 2\theta^{-1}(mn)^n
\]
is bounded above.

Let us show that $\vol(-m(K_{\oX}+\oDelta+V_\infty)|_{\oDelta})$ is also bounded from above, so that the pair $(\oX,\oDelta)$ belongs to a birationally bounded family. By Lemma \ref{lem:klt type threshold}, we know that $x\in (X,(1+c)\Delta)$ is of klt type for some constant $c=c(n,\varepsilon)>0$. By Lemma \ref{lem:cone klt type}, this implies that $-(K_V+(1+c)\Delta_V+\Delta_L)$ is big. By a similar calculation as in Lemma \ref{lem:Sasaki vol as global vol}, it follows that
\begin{align*}
    \vol(-(K_{\oX}+\oDelta+V_{\infty})|_{\oDelta}) & = \left( -(K_{\oX}+\oDelta+V_{\infty})\right)^{n-2}\cdot rV_\infty\cdot \oDelta \\
    & = r\cdot \left( -(K_V+\Delta_V+\Delta_L) \right)^{n-2}\cdot \Delta_V \\
    & \le c^{-1}r\cdot \left( -(K_V+\Delta_V+\Delta_L) \right)^{n-1} \\
    & = c^{-1}\hvol_{X,\Delta}(\ord_{V_0}) \le 2(c\theta)^{-1}n^n,
\end{align*}
where the first equality is by Lemma \ref{lem:orbifold cone property}, the second by adjunction along $V_\infty\cong V$, the next inequality by the bigness of $-(K_V+(1+c)\Delta_V+\Delta_L)$, and the last equality by Lemma \ref{lem:Sasaki vol as global vol}. Thus $\vol(-m(K_{\oX}+\oDelta+V_\infty)|_{\oDelta})$ is also bounded from above as desired. Note that $\varphi|_{\Supp(\oDelta)}$ is also birational since $\varphi$ is an embedding at $x$. Therefore, the image $(W,\Delta_W)$ of $(\oX,\oDelta)$ under the birational map induced by $|-m(K_{\oX}+\oDelta+V_\infty)|$ belongs to a fixed bounded family $(\cW,\ocD)\to B$.

By construction, $(\oX,\oDelta)$ carries an effective $\bT$-action and $V_\infty$ is $\bT$-invariant. It follows that $|-m(K_{\oX}+\oDelta+V_\infty)|$ is a $\bT$-invariant linear system and the induced birational map $\varphi$ is $\bT$-equivariant. In particular, the image $(W,\Delta_W)$ also carries an effective $\bT$-action. We claim that there exists finitely many morphisms $B_i\to B$ (depending only on the family $(\cW,\ocD)\to B$) such that
\begin{enumerate}
    \item after base change, each family $(\cW_i,\ocD_i)=(\cW,\ocD)\times_B B_i\to B_i$ admits an effective fiberwise $\bT_i$-action for some torus $\bT_i$, 
    \item $(W,\Delta_W)\cong (\cW_{b_i},\ocD_{b_i})$ for some $b_i\in B_i$, and under this isomorphism, the $\bT$-action on $W$ is induced by some group homomorphism $\bT\to \bT_i$ (in other words, the $\bT$-action on $W$ is induced by the $\bT_i$-action on $\cW_{b_i}$).
\end{enumerate}

To see this, first observe that for any torus action on a projective variety, the induced action on the Picard scheme is trivial. This is because both $\Pic^0$ (an abelian variety) and $\Pic/\Pic^0$ (a discrete group) have no $\bG_m$-action. Thus if $\cL$ is a relatively ample line bundle on $\cW$, then it is invariant under any fiberwise torus action. Consider the relative automorphism group scheme $\bG$ over $B$, which parametrizes the automorphisms of the polarized fibers $(\cW_b,\ocD_b;\cL_b)$. Note that $\bG$ is affine over $B$. Possibly after stratifying the base $B$, we may also assume that $\bG$ is smooth over $B$. By \cite[Expos\'e XII, Th\'eor\`em 1.7(a)]{SGA3-II}, after a further stratification of $B$, we may assume that the dimension of the maximal torus of $\bG_b$ is locally constant in $b\in B$. We may discard the components of $B$ where this torus dimension is zero. By \cite[Expos\'e XII, Th\'eor\`em 1.7(b)]{SGA3-II}, it then follows that there exists an (\'etale) cover $\cup B_i\to B$ of the remaining components of $B$ and a subgroup scheme $\bG_i\subseteq\bG\times_B B_i$ that reduces to the maximal tori on the fibers. Passing to a further finite cover of the $B_i$'s, we may assume that the $\bG_i$'s are split, i.e. $\bG_i=\bT_i\times B_i$ for some torus $\bT_i$. In particular, the $\bG$-action on $(\cW,\ocD)$ induces a fiberwise $\bT_i$-action on $(\cW_i,\ocD_i):=(\cW,\ocD)\times_B B_i$. This gives the family in (1). Since $(W,\Delta_W)$ has a non-trivial torus action, we see that $(W,\Delta_W)$ appears as a fiber of $(\cW_i,\ocD_i)\to B_i$ for some $i$. Since all maximal tori in $\Aut(W,\Delta_W)$ are conjugate to each other, we see that $\bT$ is conjugate to a subtorus of the maximal torus $\bT_i$. In other words, there exists an isomorphism $(W,\Delta_W)\cong (\cW_{b_i},\ocD_{b_i})$ such that (2) holds. This proves the claim.

Taking the fiberwise isolated $\bT_i$-fixed points $(\cW_i,\ocD_i)\to B_i$ for all $i$, we get families $B_i\subseteq (\cX_i,\cD_i)\to B_i$ of singularities (possibly after a refinement of the $B_i$'s), each with an effective fiberwise torus $\bT_i$-action. Since $\varphi$ restricts to an embedding on $X$, by the second part of the above claim we see that $x\in (X,\Delta;\xi)$ is isomorphic to $b_i\in (\cW_{b_i},\cD_{b_i};\xi_{b_i})$ for some $b_i\in B_i$ and some $\xi_{b_i}\in N(\bT_i)_\bR$. In particular, this gives log boundedness. Note that when $I\not\subseteq \bQ$, we can still apply the above argument to the singularities and the coefficient set constructed from Lemma \ref{lem:R to Q}, thus the same conclusion holds.

To get boundedness, we need to further stratify the family $(\cW_i,\ocD_i)\to B_i$ so that it becomes $\bR$-Gorenstein klt. By \cite[Lemma 4.44]{Kol-moduli-book} and inversion of adjunction, there exists a finite collection of locally closed subset $S_j$ of $\cup_i B_i$ such that the family becomes $\bR$-Gorenstein after base change to $\cup_j S_j$ and enumerates exactly all the klt fibers of $\cup_i (\cW_i,\ocD_i)\to \cup_i B_i$ (note that \cite[Lemma 4.44]{Kol-moduli-book} requires the family to be proper but the proof applies to our situation, essentially because the section $B_i\subseteq \cX_i$ is proper over the base). Replacing the $B_i$'s by the $S_j$'s, we obtain the desired family. The proof is now complete.
\end{proof}

\subsection{Applications}

We now explain how to deduce the other main results of this paper from Theorem \ref{thm:polarized cone bdd}. First we prove the boundedness criterion for K-semistable log Fano cone singularities.

\begin{cor} \label{cor:K-ss cone bdd}
Let $\cS$ be a set of $n$-dimensional K-semistable log Fano cone singularities with coefficients in a fixed finite set $I\subseteq [0,1]$. Then $\cS$ is bounded if and only if there exist some $\varepsilon,A>0$ such that
\[
\hvol(x,X,\Delta)\ge \varepsilon \quad \mathit{and} \quad \mldk(x,X,\Delta)\le A
\]
for all $x\in (X,\Delta)$ in $\cS$.
\end{cor}

\begin{proof}
One direction follows from Lemma \ref{lem:vol and mld^K bdd in family}, while the other direction is implied by Theorem \ref{thm:polarized cone bdd} as the volume ratio of a K-semistable log Fano cone singularity (with the K-semistable polarization) is equal to $1$.
\end{proof}

Similarly, we have an unpolarized version of Theorem \ref{thm:polarized cone bdd}.

\begin{cor} \label{cor:cone bdd unpolarized version}
Let $\cS$ be a set of $n$-dimensional log Fano cone singularities with coefficients in a fixed finite set $I\subseteq [0,1]$. Then $\cS$ is bounded if and only if there exist positive constants $\varepsilon,\theta,A>0$ such that 
\[
\hvol(x,X,\Delta)\ge \varepsilon,\quad 
\Theta(x,X,\Delta)\ge \theta,\quad\mathrm{and}\quad
\mldk(x,X,\Delta)\le A.
\]
for all $x\in (X,\Delta)$ in $\cS$.
\end{cor}

\begin{proof}
One direction follows from Lemmas \ref{lem:vol and mld^K bdd in family} and \ref{lem:density bdd in family}, while the other direction is immediate by Theorem \ref{thm:polarized cone bdd}.
\end{proof}

We next prove a version of Theorem \ref{thm:polarized cone bdd} that replaces the volume ratios with the stability thresholds introduced in \cite{Hua-thesis}. First we recall the definition. Let $x\in (X=\Spec(R),\Delta;\xi)$ be a polarized log Fano cone singularity and let $\bT$ be the torus generated by $\xi$. For this definition it would be more convenient to rescale the polarization so that $A_{X,\Delta}(\wt_\xi)=1$, which we will assume in what follows. Using the weight decomposition $R=\oplus_\alpha R_\alpha$, we set
\[
R_m:=\bigoplus_{\alpha,\,m-1<\langle \alpha,\xi\rangle\le m} R_\alpha.
\]
An $m$-basis type $\bQ$-divisor of $x\in (X=\Spec(R),\Delta;\xi)$ is defined to be a $\bQ$-divisor of the form
\[
D = \frac{1}{mN_m} \sum_{i=1}^{N_m} \{s_i=0\}
\]
where $N_m=\dim R_m$ and $s_1,\dots,s_{N_m}$ form a basis of $R_m$. Set $\delta_m=\inf_D \lct_x(X,\Delta;D)$ where the infimum runs over all $m$-basis type $\bQ$-divisors $D$. The stability threshold of $x\in (X=\Spec(R),\Delta;\xi)$ is defined as
\[
\delta(X,\Delta;\xi):=\lim_{m\to\infty} \delta_m.
\]
If $(X,\Delta;\xi)$ is the cone over a log Fano pair $(V,\Delta_V)$, then this definition is closely related to the stability threshold of $(V,\Delta_V)$ introduced in \cite{FO-delta} (see also \cite{BJ-delta}). In fact, using inversion of adjunction it is not hard to show that $\delta(X,\Delta;\xi)=\min\{1,\delta(V,\Delta_V)\}$, \emph{cf.} \cite[Theorem 3.6]{XZ-minimizer-unique}.

The stability threshold version of Theorem \ref{thm:polarized cone bdd} is a direct consequence of the following result.

\begin{lem} \label{lem:density>=delta^n}
Let $x\in (X=\Spec(R),\Delta;\xi)$ be a polarized log Fano cone singularity of dimension $n$. Then
\[
\Theta(X,\Delta;\xi)\ge \delta(X,\Delta;\xi)^n.
\]
\end{lem}

\begin{proof}
Let $\delta:=\delta(X,\Delta;\xi)$. We first recall the valuative interpretation of $\delta(X,\Delta;\xi)$ as explained in \cite{Hua-thesis}. Fix some $\bT$-invariant valuation $v\in \Val_{X,x}^*$. Set $S_m(v):=\sup_D v(D)$ where $D$ varies among $m$-basis type $\bQ$-divisors. By an Okounkov body argument \cite[Section 4]{Hua-thesis}, we know that $S(v):=\lim_{m\to\infty} S_m(v)$ exists and
\begin{equation} \label{eq:A>=delta*S}
    A_{X,\Delta}(v)\ge \delta\cdot S(v)
\end{equation}
by \cite[Theorem 4.3.5]{Hua-thesis}. We next relate $S(v)$ to the ``relative'' $S$-invariant $S(\wt_\xi;v)$ defined in \cite[Section 3.1]{XZ-minimizer-unique}. In our notation (and under our assumption that $A_{X,\Delta}(\wt_\xi)=1)$, we have $S(\wt_\xi;v)=\lim_{m\to\infty} S_m(\wt_\xi;v)$ where
\[
S_m(\wt_\xi;v) := \frac{\sum_{k=0}^m kN_k S_k(v)}{\sum_{k=0}^m kN_k}.
\]
Intuitively, the previous $S_m(v)$ is defined using basis type divisors for $R_m$ while $S_m(\wt_\xi;v)$ is defined via basis type divisors for $\oplus_{k\le m} R_k$ (the main reason for doing so in \cite{XZ-minimizer-unique} is that in the more general situation, the dimension of the analogous space $\oplus_{k\le m} R_k$ has an asymptotic expression, while the individual $R_m$ does not). Note that our $S_m(\wt_\xi;v)$ differs slightly from the one in \cite[Section 3.1]{XZ-minimizer-unique} by some round-downs, but after taking the $m\to\infty$ limit we get the same value of $S(\wt_\xi;v)$. Since $S_m(v)\to S(v)$ as $m\to\infty$, from the above expression we get $S(v)=S(\wt_\xi;v)$. From the proof of \cite[Theorem 3.7]{XZ-minimizer-unique} (especially the inequality after (3.8) in \emph{loc. cit.}), we then obtain
\[
S(v)\ge \left(\frac{\vol(\wt_\xi)}{\vol(v)}\right)^{\frac{1}{n}}.
\]
Combined with \eqref{eq:A>=delta*S} and recall that $A_{X,\Delta}(\wt_\xi)=1$, we deduce
\[
\hvol_{X,\Delta}(v) = A_{X,\Delta}(v)^n\cdot \vol(v)\ge \delta^n S(v)^n \cdot \vol(v) \ge \delta^n \vol(\wt_\xi)=\delta^n \hvol_{X,\Delta}(\wt_\xi)
\]
for all $\bT$-invariant valuation $v\in \Val_{X,x}^*$. Since the local volume of $x\in (X,\Delta)$ is computed by some $\bT$-invariant valuation (\cite{Blu-minimizer-exist} and \cite[Corollary 1.2]{XZ-minimizer-unique}), this implies that $\hvol(x,X,\Delta)\ge \delta^n \hvol_{X,\Delta}(\wt_\xi)$ and hence $\Theta(X,\Delta;\xi)\ge \delta^n$. 
\end{proof}

\begin{rem}
We also have an inequality in the reverse direction, namely, there exists some positive constant $c>0$ that only depends on the dimension such that
\begin{equation} \label{eq:delta>=Theta}
    c\cdot \delta(X,\Delta;\xi)\ge \Theta(X,\Delta;\xi).
\end{equation}
To see this, let $c=c_0^{-1}$ where $c_0$ is the constant from \cite[Lemma 3.4]{Z-mld^K-1} and let $D$ be any $m$-basis type $\bQ$-divisor of $(X,\Delta;\xi)$. Then by definition we have $\wt_\xi (D)\le 1 = A_{X,\Delta}(\wt_\xi)$, hence by \emph{loc. cit.} we see that 
\[\lct_x (X,\Delta;D)\ge c_0 \cdot \Theta(X,\Delta;\xi).\]
which gives \eqref{eq:delta>=Theta}. Note that the upper bound on the volume ratio is at least linear in the stability threshold, as can be seen on the smooth toric singularities $(X,\Delta)=(\bA^n,0)$: if $\xi=(a_1,\dots,a_n)\in\ft_\bR^+ = \bR^n_{>0}$ and $a_1\ll a_2= \dots= a_n$, then we have that $\Theta(X,\Delta;\xi) = \frac{n^n a_1\cdots a_n}{(a_1+\dots+a_n)^n}$ is roughly linear in $\delta(X,\Delta;\xi) = \frac{n a_1}{a_1+\dots+a_n}$.
\end{rem}

\begin{cor} \label{cor:cone bdd delta bdd below}
Let $n\in \bN$ and let $I\subseteq [0,1]$ be a finite set. Let $\varepsilon,\delta,A>0$. Let $\cS$ be the set of $n$-dimensional polarized log Fano cone singularities $x\in (X,\Delta;\xi)$ with coefficients in $I$ such that
\[
\hvol(x,X,\Delta)\ge \varepsilon,\quad 
\delta(X,\Delta;\xi)\ge \delta,\quad\mathrm{and}\quad
\mldk(x,X,\Delta)\le A.
\]
Then $\cS$ is bounded.
\end{cor}

\begin{proof}
This is a direct consequence of Theorem \ref{thm:polarized cone bdd} and Lemma \ref{lem:density>=delta^n}.
\end{proof}

Finally we specialize our results to dimension three. For this we need the following result from \cite{Z-mld^K-1}.

\begin{prop} \label{prop:mld^K bdd, vol>epsilon, dim=3}
Let $\varepsilon>0$ and let $I\subseteq [0,1]$ be a finite set. Then there exists some constant $A=A(\varepsilon,I)$ such that 
\[
\mldk(x,X,\Delta)\le A
\]
for all $3$-dimensional klt singularity $x\in (X,\Delta)$ with coefficients in $I$ and $\hvol(x,X,\Delta)\ge \varepsilon$.
\end{prop}

\begin{proof}
If $I\subseteq \bQ$ this is \cite[Corollary 6.11]{Z-mld^K-1}; in general we apply Lemma \ref{lem:R to Q} to reduce to the rational coefficient case.
\end{proof}

\begin{cor} \label{cor:3-dim Kss cone bdd}
For any finite set $I\subseteq [0,1]$ and any $\varepsilon>0$, the set of $3$-dimensional K-semistable log Fano cone singularities with coefficients in $I$ and with local volume at least $\varepsilon$ is bounded.
\end{cor}

\begin{proof}
Immediate from Corollary \ref{cor:K-ss cone bdd} and Proposition \ref{prop:mld^K bdd, vol>epsilon, dim=3}.
\end{proof}




The last application concerns the distribution of local volumes in dimension $3$. For any $n\in\bN$ and $I\subseteq [0,1]$, consider the set $\mathrm{Vol}^{\mathrm{loc}}_{n,I}$ of all possible local volumes of $n$-dimensional klt singularities $x\in (X,\Delta)$ with coefficients in $I$. 



\begin{cor} \label{cor:vol discrete in dim 3}
Let $I\subseteq [0,1]$ be a finite set. Then $\mathrm{Vol}^{\mathrm{loc}}_{3,I}$ is discrete away from zero.
\end{cor}

\begin{proof}
By Theorem \ref{thm:SDC}, it suffices to consider local volumes of K-semistable log Fano cone singularities. Let $\varepsilon>0$. By Corollary \ref{cor:3-dim Kss cone bdd}, the set of $3$-dimensional K-semistable log Fano cone singularities $x\in (X,\Delta)$ with coefficients in $I$ and $\hvol(x,X,\Delta)\ge \varepsilon$ is bounded. On the other hand, the local volume function is constructible in $\bR$-Gorenstein families by \cite[Theorem 1.3]{Xu-quasi-monomial} (see \cite[Theorem 3.5]{HLQ-vol-ACC} for the real coefficient case). In particular, the local volumes only take finitely many possible values in a bounded family. This implies $\mathrm{Vol}^{\mathrm{loc}}_{3,I}\cap [\varepsilon,+\infty)$ is a finite set and we are done.
\end{proof}

\begin{rem}
By combining the ideas in this work with some generalization of the proof of \cite[Theorem 1.2(2)]{HLQ-vol-ACC}, one should be able to show that if the set $I$ in Corollary \ref{cor:vol discrete in dim 3} is not finite but satisfies DCC (descending chain condition), then $\mathrm{Vol}^{\mathrm{loc}}_{3,I}$ satisfies ACC (ascending chain condition). We leave the details to the interested readers.
\end{rem}

\bibliography{ref}

\end{document}